\documentclass[12pt]{article}
\usepackage{fancyhdr}

\usepackage{amsmath, amsfonts,amssymb,  amscd, amsthm, amsxtra,verbatim}
\usepackage[all]{xy}

\DeclareMathOperator{\Spec}{Spec}
\DeclareMathOperator{\Proj}{Proj}

\DeclareMathOperator{\Hilb}{Hilb}
\DeclareMathOperator{\Prob}{Prob}
\DeclareMathOperator{\codim}{codim}

\DeclareMathOperator{\sing}{sing}
\DeclareMathOperator{\red}{red}

\DeclareMathOperator{\ch}{char}

\def\A{\mathbb{A}}
\def\Z{\mathbb{Z}}
\def\Gr{\mathbb{G}}

\def\Q{\mathbb{Q}}

\def\P{\mathbb{P}}

\def\e{\varepsilon}

\def\onto{\twoheadrightarrow}
\def\into{\hookrightarrow}
\def\xto{\xrightarrow}

\def\FF{\mathbb{F}}
\def\Z{\mathbb{Z}}

\def\d{\partial}

\theoremstyle{plain}
\newtheorem{theorem}{Theorem}[section]
\newtheorem{lemma}[theorem]{Lemma}
\newtheorem{corollary}[theorem]{Corollary}
\newtheorem{proposition}[theorem]{Proposition}

\theoremstyle{definition}

\newtheorem{conjecture}[theorem]{Conjecture}

\theoremstyle{remark}
\newtheorem{remark}[theorem]{Remark}

\begin{document}

\title{The moduli space of hypersurfaces whose singular locus has high dimension\footnote{The final publication in Mathematische Zeitschrift is available at Springer via http://dx.doi.org/[10.1007/s00209-014-1360-0]}}
\author{Kaloyan Slavov}

\maketitle

{\it Classification codes:} 51N35, 14N05, 14G15 (primary), 14Q10, 14C05, 14G05 (secondary).

\begin{abstract}
Let $k$ be an algebraically closed field and let $b$ and $n$ be integers with $n\geq 3$ and $1\leq b \leq n-1.$ Consider the moduli space $X$ of
hypersurfaces in $\mathbb{P}^n_k$ of fixed degree $l$ whose singular locus is at least $b$-dimensional. We prove that for large $l$, $X$ has a unique
irreducible component of maximal dimension, consisting of the hypersurfaces singular along a linear $b$-dimensional subspace of $\mathbb{P}^n$. The proof
will involve a probabilistic counting argument over finite fields.
\end{abstract}

\pagestyle{plain}

\section{Introduction}

Let $n$ and $b$ be fixed integers with $n\geq 3$ and $1\leq b\leq n-1$,
and let $k$ be an algebraically closed field of characteristic $p\geq 0$.
Fix a positive integer $l$.
Inside the projective space of all
hypersurfaces in $\P^n$ of degree $l$, consider the ones which are singular along some $b$-dimensional closed subscheme,
\[X=\{[F]\in\P(k[x_0,...,x_n]_l)\mid\dim V(F)_{\sing}\geq b\}\]
(this is a closed subset).

A simple argument (Lemma \ref{X^1exists}) will show that
\[X^1:=\{[F]\in X\mid L\subset V(F)_{\sing}\ \text{for some linear}\ b\text{-dimensional}\ L\subset\P^n\}\]
is an irreducible closed subset of $X$ of dimension $\binom{l+n}{n}-a_{n,b}(l),$ where
\begin{align*}
a_{n,b}(l)
&:=\binom{l+b}{b}+(n-b)\binom{l-1+b}{b}+1-(b+1)(n-b)\\
&=\frac{n-b+1}{b!}l^b+\dots.
\end{align*}

\begin{theorem}
There exists an effectively computable integer $l_0=l_0(n,b,p),$ such that for all $l\geq l_0,$ $X^1$ is the unique irreducible component of $X$
of maximal dimension.
\label{mainthm}
\end{theorem}

Our approach does not work when $l$ is small; however, we expect that the same conclusion will hold in fact for all $l\geq 2.$

In addition, $X$ has an irreducible component $X^2$ induced by the hypersurfaces singular along a $b$-dimensional quadric
(see Section \ref{seclrgcompsubsectioneopf} for the precise description of $X^2$).

\begin{theorem} Suppose that $\ch k=p>0.$
There exists an effectively computable
$l_0=l_0(n,b,p)$ such that for all $l\geq l_0,$ $X^2$ is the unique irreducible component of $X$ of second largest
dimension.
\label{corscomp}
\end{theorem}

We now sketch the main idea of the proof of Theorem \ref{mainthm}. Let $\Hilb^d$ denote the disjoint union of the finitely many Hilbert schemes
$\Hilb^{P_\alpha}_{\P^n},$ where $P_\alpha$ ranges over the Hilbert polynomials of integral $b$-dimensional
closed subschemes $C\subset\P^n$
of degree $d$, and define the restricted Hilbert scheme $\widetilde{\Hilb}^d$ as
the closure in $\Hilb^d$ of the set of points corresponding to integral subschemes. Let
$V=k[x_0,...,x_n]_l$. Consider the incidence correspondence
\[\widetilde{\Omega}^d=\{(C,[F])\in\widetilde{\Hilb}^d\times\P(V)\mid C\subset V(F)_{\sing}\}.\]
The first step\footnote{We are going to be slightly imprecise here; see Section \ref{smldegsubsection} for the exact statement.}
is to show that in the case of ``small" degree $2\leq d\leq B$ (for an appropriately chosen $B$),
any irreducible component of $\widetilde{\Omega}^d$ has dimension less than $\dim X^1.$ For this, we apply the theorem on dimension
of fibers to the map $\pi\colon \widetilde{\Omega}^d\to\widetilde{\Hilb}^d$. A result of Eisenbud and Harris gives
$\dim\widetilde{\Hilb}^d$ when $b=1$;
in fact, a naive rude bound valid for all $b$ is sufficient for our purposes. So it remains to give
an upper bound for the dimension of the fiber of  $\pi$ over an integral $C$ of degree $d$. For this, we
specialize $C$ to a union of $d$ $b$-dimensional linear subspaces that contain a common $(b-1)$-dimensional linear
subspace.

The second step is to handle the case of ``large" degree $d\geq B+1$. For this, the first main observation is that it
suffices to assume that $k=\overline{\FF_p}$ in the statement of the main theorem. The reason is that
the variety $X$ is the basechange by $\Spec k\to\Spec\Z$ of a projective variety $X^{\text{univ}}\to\Spec\Z$, and in order to
give an upper bound for $\dim X^{\text{univ}}\times{\overline{\Q}},$ by upper-semicontinuity,
it suffices to give an upper bound for $\dim X^{\text{univ}}\times{\overline{\FF_p}}$
for a single prime $p$ (we will take $p=2$).

So let $k=\overline{\FF_p}$ and $d\geq B+1.$ We have to give an upper bound for the dimension of
\begin{multline*}
T^d=\{[F]\in\P(k[x_0,...,x_n]_l)\mid V(F)_{\sing}\ \text{contains}\\ \text{a subscheme with Hilbert polynomial among $\{P_\alpha\}$}\}.
\end{multline*}
Any variety $T$ over $\overline{\FF_p}$ comes from a variety $T_0$ defined over some finite field $\FF_{q_0}$; in order to
give an upper bound $\dim T\leq A$, it suffices to prove that $\#T_0(\FF_q)=O(q^A)$ as $q\to\infty,$ by the result of
Lang-Weil \cite{LW}. So we reduce the problem to giving an upper bound on the number of hypersurfaces
$F\in\FF_q[x_0,...,x_n]_l$ such that $V(F)_{\sing}$ contains an integral closed subscheme of large degree $d$.

For this, we mimic the main argument in \cite{P}. We sketch it here in the case $b=1$ and $l\equiv 1\pmod{p}$ to simplify notation.
Write $F$ in the form
\[F=F_0+\sum_{i=0}^n G_i^p x_i,\]
where $F_0$ has degree $l,$ each $G_i$ has degree $\tau=\frac{l-1}{p},$ and note that
\[\frac{\d F}{\d x_i}=\frac{\d F_0}{\d x_i}+G_i^p.\]
Fix $F_0$. We exhibit a large supply of $(G_0,...,G_n)$ such that the $F$ constructed in this way has the property that
$V(F)_{\sing}$ contains no integral curves of degree $d$. To do this, we first give a large supply of
 $(G_0,...,G_{n-2})$ such that
$V(\frac{\d F}{\d x_0},...,\frac{\d F}{\d x_{n-2}})$ has all components of dimension $1$.
The number of such components is bounded by B\'ezout's theorem.
It remains to give a large supply of $G_{n-1}$ such that no irreducible
component $C$ of $V(\frac{\d F}{\d x_0},...,\frac{\d F}{\d x_{n-2}})$
of degree $d$ is contained in $V(\frac{\d F_0}{\d x_0}+G_{n-1}^p)$.
We accomplish this by specializing $C$ to a union of $d$ lines again,
and giving an upper bound on the number of $G_{n-1}$ with $C\subset V(\frac{\d F_0}{\d x_0}+G_{n-1}^p)$.
With some technical details concerning the uniqueness of the
largest-dimensional component in characteristic $0$, this completes the proof of Theorem \ref{mainthm}. The discussion of the
second largest component is along the same lines.

\section{Notation}

For a field $k$, the graded ring $k[x_0,...,x_n]$ will be denoted by $S$.
For a graded $S$-module $M$ (in particular, for a homogeneous ideal),
$M_l$ will denote the $l$-th graded piece of $M$. When $I\subset S$ is a homogeneous ideal, $(I^2)_l$ is denoted simply by $I^2_l.$
Also, $k[x_0,...,x_n]_{\leq l}$ denotes the vector space of (inhomogeneous) polynomials whose total degree is at most $l$.
When the field $k$ and the integer $l$ are fixed, $V$ will denote the vector space $V=k[x_0,...,x_n]_l.$

For a finite-dimensional $k$-vector space $V$, $\P(V)$ denotes the projective space parametrizing lines in $V$.
 Given a homogeneous ideal $I\subset k[x_0,...,x_n],$ $V(I)$ denotes the closed subscheme
$\Proj (k[x_0,...,x_n]/I)\into\P^n_k$, and for $i=0,...,n$, $D_+(x_i)$ is the complement of $V(x_i).$
We often abbreviate $V(\{G_i\}_{i\in I})\subset\P^n$ as $V(G_i),$ when the index set $I$ is irrelevant or understood.

For $F\in S_l$, $V(F)_{\sing}\subset\P^n$ is the closed subscheme
$V(F,\frac{\d F}{\d x_i})=V(F,\frac{\d F}{\d x_0},...,\frac{\d F}{\d x_n})$ of $\P^n$, so when $F\neq 0,$ the underlying topological
space of $V(F)_{\sing}$ is the singular locus of $V(F)$.

We will reuse $l_0$ for different bounds as we go along, in order to avoid unnecessary notation; however,
it will be clear that we are actually referring to different values of $l_0$ even though we use the same symbol. Also, it will
be understood that sometimes the value of $l_0$ is the maximum of a finite set of previously defined bounds, each of them
still denoted by $l_0.$

For integers $b$ and $n$ with $1\leq b\leq n-1,$ we denote by $\Gr(b,n)$ the Grassmanian of $b$-dimensional projective linear
subspaces of $\P^n.$

\section{The incidence correspondence and the restricted Hilbert scheme}
\label{functorZ}

The first goal of this section is to introduce a universal incidence correspondence $\Omega^P$
over $\Spec\Z$ and the universal moduli spaces $T^P\to\Spec\Z$.
Secondly, we introduce the restricted Hilbert scheme and discuss an
upper bound for its dimension.

\subsection{The incidence correspondence}
\label{inc_corr_subsection}

Let $l\geq 1$ be an integer, and let $V=\Z[x_0,...,x_n]_l$ now. For an algebraically closed field $k$, set $V_k=V\otimes_{\Z}k=k[x_0,...,x_n]_l$.
Consider a polynomial $P\in\Q[z],$ and let $\Hilb^P_{\P^n_{\Z}}$ be the Hilbert scheme corresponding to $P$.

Using standard arguments, one can show that there exists a closed subscheme $\Omega^P$ of $\Hilb^P_{\P^n_{\Z}}\times\P(V)$ such that for any algebraically closed field $k$, the basechange
$\Omega^P_k=\Omega^P\times\Spec k$ is given by
\[\Omega^P_k=\left\{(C,[F])\in \Hilb_{\P^n_k}^P\times\P(V_k)\
|\ C\subset V\left(F,\frac{\partial F}{\partial x_i}\right)\right\}\]
(inclusion above denotes scheme-theoretic inclusion).

Let $T^P$ denote the scheme-theoretic image of $\Omega^P\to\P(V).$
For any algebraically closed field $k$, we obtain
a diagram
\[\xymatrix{
\Omega^P_k\ar@{->>}[d]\ar@{^{(}->}[r] & \Hilb^P_{\P^n_k}\times\P(V_k)\ar[d]\\
T^P_k\ar@{^{(}->}[r] &\P(V_k)
}
\]
and so
\[T^P_k=\{[F]\in\P(V_k)\mid V(F)_{\sing}\ \text{contains a subscheme with Hilbert polynomial}\ P\}.\]

Next, we recall (see \cite{B}, p.~3) the following classical result.

\begin{theorem}[Chow's finiteness theorem] Fix positive integers $n,b,d.$ There are only finitely many Hilbert polynomials $P_\alpha$ of
integral $b$-dimensional closed subschemes of $\P^n_k$ of degree $d$. The algebraically closed field $k$ varies as well in this
statement.
\end{theorem}

For an integer $d\geq 1,$
let $\Hilb^{b,d}_{\P^n}$ be the disjoint union of the Hilbert schemes
$\Hilb^{P_\alpha}_{\P^n}$ for all the finitely many possible Hilbert polynomials $P_{\alpha}$
of an integral $b$-dimensional closed subscheme $C\subset\P^n$ of degree $d$.
Let $\Omega^d$ be the disjoint union of the finitely many $\Omega^{P_{\alpha}}$.
Also, define $T^d$ as the scheme-theoretic image of $\Omega^d\to\P(\Z[x_0,...,x_n]_l).$ For any algebraically closed field $k$, the basechange $T^d_k$ is described by
\begin{multline*}T^d_k=\bigcup T^{P_\alpha}_k=\{[F]\in\P(V_k)\mid V(F)_{\sing}\ \text{contains}\\
 \text{a subscheme with Hilbert polynomial among $\{P_\alpha\}$}\}.
\end{multline*}
Any integral closed subscheme of degree $1$ is linear, so
$X^1=T^1_k$. We will use $X^1$ and $T^1_k$ interchangeably.

Over an algebraically closed field $k$, define
\[\hat{T}^d:=T^d_k-\left(T^d_k\cap \left(\bigcup_{d'=1}^{d-1} T^{d'}_k\right)\right).\]

\subsection{The restricted Hilbert scheme}
\label{gather_defs}

 Let $k$ be an algebraically closed field.
Define the restricted Hilbert scheme $\widetilde{\Hilb}^{b,d}_{\P^n}$ to be the Zariski
closure in $\Hilb^{b,d}_{\P^n}$ of the set of integral subschemes, with reduced subscheme structure.

Eisenbud and Harris \cite{EH} prove the following result for the dimension of
$\widetilde{\Hilb}^{b,d}_{\P^n}$ in the case $b=1$.

\begin{theorem}
Let $b=1$. For $d\geq 2,$ the largest irreducible component of $\widetilde{\Hilb}^{1,d}_{\P^n}$ is the one corresponding to the
family of plane curves of degree $d$; in particular,
$\dim\widetilde{\Hilb}^{1,d}_{\P^n}=3(n-2)+\frac{d(d+3)}{2}.$
\label{EHthmdim}
\end{theorem}

In analogy, for $b\geq 2,$ Eisenbud and Harris state the following conjecture:

\begin{conjecture}
For $d\geq 2,$ the largest irreducible component of $\widetilde{\Hilb}^{b,d}_{\P^n}$ is the one corresponding to the
family of degree-$d$
hypersurfaces contained in linear $(b+1)$-dimensional subspaces of $\P^n$; in particular,
$\dim\widetilde{\Hilb}^{b,d}_{\P^n}=(b+2)(n-b-1)-1+\binom{d+b+1}{b+1}$.
\label{dim_Hil}
\end{conjecture}

However, an easy rude bound for $\dim\widetilde{\Hilb}^{b,d}_{\P^n}$ is provided by the following

\begin{lemma}
\[\dim \widetilde{\Hilb}^{b,d}_{\P^n}\leq (n+1)\left(\binom{d+n}{n}-1 \right).\]
\label{easy_naive_dim_bound_Hlbrt_scheme}
\end{lemma}

\begin{proof}
Let $S_d=k[x_0,...,x_n]_d$ and let
\[Q^b=\{ ([f_0],...,[f_n])\in \P(S_d)\times ...\times\P(S_d)\ |\ \dim V(f_0,...,f_n)=b\};\]
this is a constructible subset of $\P(S_d)^{n+1}.$
Theorem 9.7.7 in \cite{EGAIV.3}, applied to the universal family $\Theta\into\P^n\times {\Hilb^{b,d}_{\P^n}}\to
\Hilb^{b,d}_{\P^n}$, implies that
\[\widehat{\Hilb}^{b,d}_{\P^n}:=\{Y\in\Hilb^{b,d}_{\P^n} \ |\
Y\ \text{is integral}\}\] is constructible.

Consider the incidence correspondence
\[\mathcal{A}:=\{ (Y,([f_0],...,[f_n]))\in\widehat{\Hilb}^{b,d}_{\P^n}\times Q^b \ \vert\ Y\into V(f_0,...,f_n)\}\]
with its two projections $\pi_1$ and $\pi_2$ to
$\widehat{\Hilb}^{b,d}_{\P^n}$ and $Q^b$, respectively.
The fibers under $\pi_2$ are finite, since for a fixed
$([f_0],...,[f_n])\in Q^b,$ $Y$ will have to be one of the finitely many irreducible components of
$V(f_0,...,f_n)_{\text{red}}$. Therefore,
$\dim \mathcal{A}\leq\dim Q^b\leq\dim\P(S_d)^{n+1}.$
On the other hand, Exercise I.3.28.2 in \cite{Kollar}
implies that $\pi_1$ is surjective, and the dimension bound  follows.
\end{proof}

When $b$ and $n$ are understood, we abbreviate
$\widetilde{\Hilb}^{b,d}_{\P^n}$ as $\widetilde{\Hilb}^d$. We now give some more related definitions, which will be used later. For a polynomial $P(z)\in\Q[z],$ define
\[\widetilde{\Omega}^P=\{(C,[F])\in\widetilde{\Hilb}^P\times\P(V)\mid C\subset V(F)_{\sing}\},\]
and for an integer $d\geq 1,$ define
\[\widetilde{\Omega}^d=\{(C,[F])\in\widetilde{\Hilb}^d\times\P(V)\mid C\subset V(F)_{\sing}\}\]
(as always, inclusions are scheme-theoretic).
Also, let
\[R^d=\{(C,[F])\in\widetilde{\Hilb}^d\times\P(V)\mid C\ \text{is integral,}\  C\subset V(F)_{\sing}\}\subset \widetilde{\Omega}^d.\]
and let $\overline{R^d}$ be the closure of $R^d$ inside $\widetilde{\Omega}^d$ (or inside $\widetilde{\Hilb}^d\times\P(V)$).

\section{Specialization arguments}
\label{specargsection}

The first main technique that we use in the proof of Theorem \ref{mainthm} is a specialization argument, that allows us to
bound $\dim\{F\in k[x_0,...,x_n]_l\mid C\subset V(F)_{\sing}\}$ from above for a fixed $C$, by degenerating $C$
to a union of linear spaces.
In Section \ref{speclemmapf}, we prove (for lack of reference)
that we can specialize a $b$-dimensional integral closed subscheme $C$ of $\P^n$ to a union of
$d$ $b$-dimensional linear spaces containing a common $(b-1)$-dimensional linear space. Next, the bound we obtain in
Section \ref{subsspecargsing}
will be the main ingredient for the discussion of the cases of small degree $2\leq d\leq B$ in
Chapter \ref{smalldegreesection}. Finally, Section \ref{specCsubVF} is a preparation for the discussion of the case of large
degree $d\geq B+1$, which will be treated in Chapter \ref{lrg_d_p}. The main result of Section \ref{specCsubVF} is stated
in Corollaries \ref{spdmprob} and \ref{spSprob} in a form that is most convenient for later purposes.

In this section, $k$ is a fixed algebraically closed field.

\subsection{Specialization of a closed subscheme to a union of linear subspaces}
\label{speclemmapf}

The result of this section is known, but we were unable to find a reference, so we include it here.

Let $C\subset\P^n$ be an integral $b$-dimensional closed subscheme of degree $d$.
Let $P=V(x_0,...,x_{n-b})$ be the $(b-1)$-dimensional ``linear subspace at infinity."
Suppose that the linear subspace $H=V(x_{n-b+1},...,x_n)$ intersects $C$
in $d$ distinct points $Q_i.$ Let $L_i$ be the unique $b$-dimensional linear space through $P$ and $Q_i$ (note that $L_i\neq L_j$ for $i\neq j$).
 Consider the projective linear transformations
\[A_a=\left(\begin{array}{ccc|ccc}
a & {} &{} &{}&{}&{}\\
{}&\ddots &{}&{}&{}&{}\\
{}&{}&a&{}&{}&{}\\ \hline
{}&{}&{}&1&{}&{}\\
{}&{}&{}&{}&\ddots&{}\\
{}&{}&{}&{}&{}&1
\end{array}
\right)\]
(where the bottom block has size $b\times b$) and let $C_a=A_aC$.

\begin{proposition}
The underlying topological space of the flat limit $C_0=\lim_{a\to 0}C_a$ is $\bigcup_{i=1}^d L_i.$
\label{specprop}
\end{proposition}

\begin{proof}
Let $C=V(\{G_s\})\subset\P^n$ (as a scheme), where $G_s\in k[x_0,...,x_n]$ are homogeneous.
Consider the map
\[\sigma\colon \P^n\times (\A^1-\{0\})\to \P^n,\quad ([x_0,...,x_n],a)\mapsto (x_0,...,x_{n-b},ax_{n-b+1},...,ax_n),\]
and define the closed subscheme $X\subset \P^n\times(A^1-\{0\})$ as the fiber product
\[\xymatrix{
X\ar@{^{(}->}[r]\ar[d] & \P^n\times (\A^1-\{0\})\ar[d]_{\sigma}\\
C\ar@{^{(}->}[r] & \P^n.
}
\]
In other words,
\[X=V(G_s(x_0,...,x_{n-b},ax_{n-b+1},...,ax_n))\subset \P^n_{\A^1-\{0\}},\]
where we regard $G_s(x_0,...,x_{n-b},ax_{n-b+1},...,ax_n)\in k[a,a^{-1}][x_0,...,x_n]$.
 This is a flat family $X\to\A^1-\{0\},$ whose fiber over
$a\neq 0$ is $C_a$ (as a subscheme of $\P^n$).

Let $\overline{X}$ be the scheme-theoretic closure of $X$ in $\P^n\times\A^1$.
By the proof of Proposition III.9.8 in \cite{H},
the flat limit of the family $(C_a)$ is the scheme-theoretic fiber $\overline{X}_0.$

Consider
\[Y=V(G_s(x_0,...,x_{n-b},ax_{n-b+1},...,ax_n))\subset\P^n\times\A^1.\]
Then $Y$ is a closed subscheme of $\P^n\times\A^1$ containing $X_0$ (scheme-theoretically), so $Y$ contains $\overline{X}.$
Thus, $\overline{X}_0\subset Y_0$ is a closed subscheme.

\[\xymatrix{
{} & {} & Y\ar[ldd] \\
X\ar[r]\ar@{^{(}->}[d]\ar[rru] & \overline{X}\ar@{^{(}->}[d]\ar@{.>}[ru]\\
\P^n\times (\A^1-\{0\})\ar@{^{(}->}[r]\ar[d] & \P^n\times\A^1\ar[d] \\
\A^1-\{0\}\ar@{^{(}->}[r] & \A^1
}
\]

We have
\[Y_0=V(G_s(x_0,...,x_{n-b},0,...,0))\subset\P^n.\]
Thus, as a set, $Y_0$ is $\bigcup_{i=1}^d L_i$.

By the assumption that $C$ and $H$ intersect transversely, we deduce that
$Y_0$ is reduced away from $P$ (in general, it could be nonreduced along $P$).
It follows that the Hilbert polynomial of $Y_0$ has the same
degree and leading coefficient (namely, $b$ and $d/b!$, respectively) as the Hilbert polynomial of $(Y_0)_{\red}$.
The
Hilbert polynomial of the flat limit $\overline{X}_0$ also has degree $b$ and leading coefficient $d/b!$.
Moreover, $Y_0$ is equidimensional, so
 the inclusion $\overline{X}_0\into Y_0$ must
be a homeomorphism.
\end{proof}

Let $V=k[x_0,...,x_n]_l.$ For each closed subscheme $C\subset\P^n$, define the $k$-vector space
\[W_C=\{F\in V\mid C\subset V(F)_{\sing}\}.\]

\begin{corollary} Let $C\into\P^n$ be an integral closed subscheme of dimension $b$ and degree $d$. There exist $d$ $b$-dimensional
linear subspaces $L_1,...,L_d$ of $\P^n$ containing a common $(b-1)$-dimensional linear subspace, such that
\[\dim W_C\leq \dim W_{\cup L_i},\]
where $\cup L_i$ is given the reduced induced structure.
\end{corollary}

\begin{proof}
Let $P$ be the Hilbert polynomial of $C$.
Apply the upper semicontinuity theorem (see Section 14.3 in \cite{E}) to the map
\[\{(C,[F])\in\Hilb^P\times \P(V)\mid C\subset V(F)_{\sing}\}\xto{\pi} \Hilb^P.\]
By Proposition \ref{specprop}, $\cup L_i$ (with some scheme structure) is the flat limit $C_0$ of a family $(C_a)$,
 with each $C_a$ ($a\neq 0$) being projectively
equivalent to $C=C_1,$ and hence
$\pi^{-1}(C_a)\simeq\pi^{-1}(C)$ for each $a\neq 0.$
Therefore,
\[\dim \P(W_C)=\dim \pi^{-1}(C)\leq\dim \pi^{-1}(C_0)=\dim\P(W_{C_0})\leq\dim\P(W_{\cup L_i}).\qedhere\]
\end{proof}

\subsection{An upper bound on the dimension of the space of $F$ such that
$C\subset V(F)_{\sing}$, for a fixed $C$ of small degree}
\label{subsspecargsing}

Fix a positive integer $l$. Recall the notation $V=k[x_0,...,x_n]_l.$

\begin{lemma}
Let $L\subset\P^n$ be a $b$-dimensional linear subspace. Then for $F\in V$, we have $L\subset V(F)_{\sing}$ if and only if
$F\in I_L^2$. Moreover,
\[\codim_V \{F\in V\mid L\subset V(F)_{\sing}\}=\binom{l+b}{b}+(n-b)\binom{l-1+b}{b}.\]
\label{cod_one_lin}
\end{lemma}

\begin{proof}
Without loss of generality, $L=V(I)$ with $I=(x_{b+1},...,x_n).$ For $F\in V$, we check that $(F,\frac{\d F}{\d x_i})\subset I$ if
and only if
$F\in I^2$.
Clearly, $(S/I^2)_l\simeq k[x_0,...,x_b]_l\oplus(\bigoplus_{i=b+1}^n k[x_0,...,x_b]_{l-1}x_i)$ has dimension as in the statement.
\end{proof}

\begin{lemma} Let $L_1,...,L_d$ be $d$ $b$-dimensional linear subspaces of $\P^n$ containing
a common $(b-1)$-dimensional linear subspace. Then
for $d\leq \frac{l+1}{2},$ we have
\[\codim_V(W_{\cup L_i})\geq \binom{l+b}{b}+(n-b)\sum_{e=1}^d\binom{l-2e+1+b}{b}.\]
\label{codi_d_spaces}
\end{lemma}

\begin{proof}
We induct on $d$. For $d=1$, we have equality. Assume $2\leq d\leq\frac{l+1}{2}.$
Assume that the $b$-dimensional linear subspaces $L_1,...,L_d$ all contain
$P=[0,\underbrace{*,...,*}_{b},0,...,0]$
and that none of them is contained in the hyperplane $x_0=0,$ so the ideal
of each of them is of the form $(x_{b+1}-p_{b+1} x_0,...,x_n-p_n x_0)$
for a uniquely determined tuple $(p_{b+1},...,p_n)\in k^{n-b}$. Let
\[I_i=(x_{b+1}-p_{b+1}^{(i)}x_0,x_{b+2}-p_{b+2}^{(i)}x_0,...,x_n-p_n^{(i)}x_0)\quad\text{for}\ i=1,...,d-1,\]
and without loss of generality
\[I_d=(x_{b+1},...,x_n).\]
By Lemma \ref{cod_one_lin}, $W_{\cup L_i}=(I_1^2\cap \dots\cap I_d^2)_l,$ so we have to give a lower bound for
$\dim(S/I_1^2\cap\dots\cap I_d^2)_l.$ For $e\in\{d-1,d\},$ let $\mu_e=\dim (S/I_1^2\cap\dots\cap I_e^2)_l.$
There is a short exact sequence
\[0\to \left(\frac{I_1^2\cap\dots\cap I_{d-1}^2}{I_1^2\cap \dots\cap I_d^2}\right)_l\to
\left(\frac{S}{I_1^2\cap\dots\cap I_d^2}\right)_l\to \left(\frac{S}{I_1^2\cap\dots\cap I_{d-1}^2}\right)_l\to 0.\]

For each $i=1,...,d-1,$ there exists $m_i\in\{b+1,...,n\}$ such that $p_{m_i}^{(i)}\neq 0.$ Let
$F=\prod_{i=1}^{d-1}(x_{m_i}-p^{(i)}_{m_i}x_0)^2.$ Consider all elements
\[Fx_j P(x_0,...,x_b)\in \left(\frac{I_1^2\cap\dots\cap I_{d-1}^2}{I_1^2\cap \dots\cap I_d^2}\right)_l,\]
where $j\in\{b+1,...,n\}$ and $P(x_0,...,x_b)$ runs through a basis of $k[x_0,...,x_b]_{l-2d+1}.$
Their number is $(n-b)\binom{l-2d+1+b}{b}$ and they are linearly independent.

Therefore
\[\mu_d\geq \mu_{d-1}+(n-b)\binom{l-2d+1+b}{b},\]
and the statement follows by induction.
\end{proof}

\subsection{An upper bound on the dimension of the space of $F$ such that
$C\subset V(F)$, for a fixed $C$ of small degree}
\label{specCsubVF}

\begin{lemma} Fix positive integers $l,m,$ with $m\leq l+1.$ For any integral closed subscheme $C\subset \P^n$ of dimension $b$ and
degree $d\geq m,$ we have
\[\codim_V\{G\in V\mid C\subset V(G)\}\geq \sum_{e=1}^m\binom{l-e+1+b}{b}=:A_{b}(l,m).\]
\label{cont_C_lrg_deg}
\end{lemma}

\begin{proof}
As above, we specialize $C$ to a union of $d$ $b$-dimensional linear spaces containing $P$ (notation as in the previous lemma).
Throwing away some of these linear spaces if necessary, we may assume $d=m.$ So we induct on $m=1,...,l+1$ to give a lower
bound for $\dim (S/I_1\cap\dots\cap I_m)_l$.
We follow the notation and proof of Lemma \ref{codi_d_spaces},
except that this time, we consider
$F=\prod_{i=1}^{m-1}(x_{m_i}-p^{(i)}_{m_i}x_0)$ and the linearly independent elements
$FP\in (I_1\cap\dots\cap I_{m-1}/I_1\cap\dots\cap I_m)_l,$ where $P$ runs through a basis of $k[x_0,...,x_b]_{l-m+1}.$
\end{proof}

\begin{corollary} Let $k$ be an algebraically closed field, and $k_0\subset k$ a subfield. Again, let $m,l$ be fixed integers,
with $m\leq l+1.$ Let $C\subset\P^n_k$ be a $b$-dimensional integral closed subscheme
(not necessarily defined over $k_0$) of degree $d\geq m.$ Then
\[\codim_{k_0[x_0,...,x_n]_l}\{G\in k_0[x_0,....,x_n]_l \mid C\subset V(G)\}\geq A_b(l,m).\]
Here, the condition $C\subset V(G)$ (inclusion of closed subschemes of $\P^n_{k}$)
 makes sense when we regard $G\in k[x_0,...,x_n]_l$ first.
\label{containcurve}
 \end{corollary}

\begin{proof}
It suffices to prove that
\[\dim_{k_0}\{G\in k_0[x_0,...,x_n]_l\mid C\subset V(G)\}\leq\dim_k\{G\in k[x_0,...,x_n]_l\mid C\subset V(G)\}.\]
This is automatic, since
any $k_0$-linearly independent elements in $k_0[x_0,...,x_n]_l$ are $k$-linearly independent
in $k[x_0,...,x_n]_l$.
\end{proof}

\begin{corollary} Let $k_0=\FF_q$ now. Let $C\subset \P^n_{\overline{\FF_p}}$ be an integral
 $b$-dimensional closed subscheme of degree $d\geq m$ (again, $m\leq l+1$ is fixed).
For $G$ chosen randomly from $\FF_q[x_0,...,x_n]_l,$ we have
\[\Prob(C\subset V(G))\leq q^{-A_b(l,m)}.\]
\label{spdmprob}
\end{corollary}

\begin{proof}
This is just a restatement of Corollary \ref{containcurve}, since
\[\#\{G\in \FF_q[x_0,...,x_n]_l\mid C\subset V(G)\}=q^{\text{dim}\{G\mid C\subset V(G)\}}.\qedhere\]
\end{proof}

\begin{lemma}
Let $k$ be an algebraically closed field, and $D\subset\P^n_k$ an integral closed subscheme of dimension at least $b+1$. Then
\[\codim_{k[x_0,...,x_n]_l}\{G\in k[x_0,...,x_n]_l\mid D\subset V(G)\}\geq \binom{l+b+1}{b+1}.\]
\label{bplusonean}
\end{lemma}

\begin{proof}
We can assume that $\dim D=b+1.$
This is a particular case of Lemma \ref{cont_C_lrg_deg}; just note that $A_{b+1}(l,1)=\binom{l+b+1}{b+1}.$
\end{proof}

The same argument leading from Lemma \ref{cont_C_lrg_deg} to Corollary \ref{spdmprob} leads from Lemma
\ref{bplusonean} to the following

\begin{corollary}
Let $k_0=\FF_q$ now. Let $D\subset \P^n_{\overline{\FF_p}}$ be an integral closed subscheme of dimension at least $b+1$.
For $G$ chosen randomly from $\FF_q[x_0,...,x_n]_l,$ we have
\[\Prob(D\subset V(G))\leq q^{-\binom{l+b+1}{b+1}}.\]
\label{spSprob}
\end{corollary}

\section{The case of small degree $d$}
\label{smalldegreesection}

With the preparations from the previous section, it is now easy to handle the cases of small degree $2\leq d\leq B$
by applying the theorem on the dimension of fibers to the map
$\widetilde{\Omega}^d\to\widetilde{\Hilb}^d$ (Section \ref{smldegsubsection}).
The main result of this section is Corollary \ref{sml_m}.
 Finally, in Section \ref{seccompsmldegsection}, we
perform the analogous calculation for the second largest component of $X$.

Again, $k$ is a fixed algebraically closed field.

\subsection{The component corresponding to $d=1$}

Recall the definitions of $X^1$ and $a_{n,b}(l)$ from the introduction.

\begin{lemma} The set
$X^1$ is an irreducible closed subset of $X$ of dimension equal to $A:=\binom{l+n}{n}-a_{n,b}(l).$
\label{X^1exists}
\end{lemma}

\begin{proof}
Consider
\[\Omega^1=\{(L,[F])\in \Gr(b,n)\times\P(V)\mid L\subset V(F)_{\sing}\}\subset \Gr(b,n)\times\P(V).\]

Let
$\pi\colon \Omega^1\to \Gr(b,n)$ and $\rho\colon \Omega^1\to \P(V)$ denote the two projections.
The fiber of $\pi$ over any linear $b$-dimensional $L$ is $\P(W_{L}).$ So $\Omega^1$
is irreducible, and has dimension $\dim \P(W_{L})+\dim \Gr(b,n)=A$ (use Lemma \ref{cod_one_lin}).

Consider now $\rho\colon \Omega^1\onto X^1.$
To prove that $\Omega^1$ and $X^1$ have the same dimension, it suffices to show that some fiber of $\rho$ is $0$-dimensional.
This is easy (we prove a more general statement later; see Lemma \ref{ctrlsng}).
\end{proof}

\subsection{The case $d\leq B$ (small degree)}
\label{smldegsubsection}

Fix an integer $l$ as usual, and fix an integer $d>1.$ As usual, let $V=k[x_0,...,x_n]_l.$ Recall the definitions of $R^d$ and $\overline{R^d}$ from Section \ref{gather_defs}.

Let $\pi\colon \widetilde{\Hilb}^d\times\P(V)\to\widetilde{\Hilb}^d$ and
$\rho\colon \widetilde{\Hilb}^d\times\P(V)\to\P(V)$ denote the first and second projections.

\begin{lemma} Fix an integer $B$. There exists $l_0$ (effectively computable) such that for all pairs $(d,l)$ with
$2\leq d\leq B$ and  $l\geq l_0,$
we have
\[\dim \overline{R^d}<\dim X^1.\]
It follows that $\dim\rho(\overline{R^d})<\dim X^1.$
\label{sml}
\end{lemma}

\begin{proof}
Let $Z$ be an irreducible component of $\overline{R^d}.$
Certainly, $Z\cap R^d\neq\emptyset$, so $\pi(Z)$ contains an integral subscheme $C\subset\P^n$.
Degenerate $C$ to a union $\bigcup_{i=1}^d L_i$ of $d$ $b$-dimensional linear spaces, as in Section \ref{speclemmapf}. Let $L_0$
be any linear $b$-dimensional subspace of $\P^n$.
By abuse of notation, let
$\pi\colon Z\onto\pi(Z)\subset\widetilde{\Hilb}^d.$ By the theorem on the dimension of fibers, we have
\begin{align}
\dim Z &\leq \dim \pi^{-1}(C)+\dim \pi(Z)\notag\\
&\leq \dim\P(W_C)+\dim\pi(Z)\notag\\
&\leq \dim\P(W_{\cup L_i})+\dim\widetilde{\Hilb^d}.
\label{dimfiberineq}
\end{align}
Thus, it suffices to check that
\[\dim \P(W_{\cup L_i})+\dim \widetilde{\Hilb}^d<\dim \P(W_{L_0})+(b+1)(n-b)\]
(recall Lemma \ref{X^1exists}),
or, equivalently, that
\[\codim_V W_{L_0}+\dim\widetilde{\Hilb}^d<\codim_V W_{\cup L_i}+(b+1)(n-b).\]

By Lemmas \ref{cod_one_lin} and \ref{codi_d_spaces}, it suffices to prove the inequality
\begin{align*}
\binom{l+b}{b}+(n-b)\binom{l-1+b}{b}+\dim\widetilde{\Hilb}^d\\
<\binom{l+b}{b}+(n-b)\sum_{e=1}^d\binom{l-2e+1+b}{b}+(b+1)(n-b),
\end{align*}
or, equivalently,
\begin{equation}
\dim\widetilde{\Hilb}^d-(b+1)(n-b)<(n-b)\sum_{e=2}^d \binom{l-2e+1+b}{b},
\label{csmuncond}
\end{equation}
for all $2\leq d\leq B$ and $l\geq l_0$.

When $d\in\{2,...,B\}$ is fixed, this inequality is satisfied for $l\gg 0.$ Moreover, using the naive bound for
$\dim \widetilde{\Hilb}^d$ given in Lemma \ref{easy_naive_dim_bound_Hlbrt_scheme}, we can compute an effective bound for $l.$
\end{proof}

\begin{remark} If we use the Eisenbud--Harris bound for $\dim \widetilde{\Hilb}^d$ from Theorem \ref{EHthmdim} (when $b=1$) and Conjecture \ref{dim_Hil}
(when $b\geq 2$), we can strengthen the statement of the previous Lemma and show that there exists $l_0$ (effectively computable) such that for all pairs $(d,l)$ with
$2\leq d\leq\frac{l+1}{2}$ and  $l\geq l_0,$
we have
\[\dim \overline{R^d}<\dim X^1.\]
Of course, using the Eisenbud--Harris bounds will yield better bounds on $l_0$. We will follow the approach of using instead the naive bound of Lemma \ref{easy_naive_dim_bound_Hlbrt_scheme}
for the dimension of the
restricted Hilbert scheme in order to avoid conditional statements when $b=2$.
\label{sml_better_bound}
\end{remark}

\begin{corollary}
Fix an integer $B$, and let $l_0$ be as in Lemma \ref{sml}. Let $2\leq d\leq B$ and $l\geq l_0$.
If $Z\subset T^d_k$ is an irreducible component, then either $Z=X^1$, or
$\dim Z<\dim X^1.$
\label{sml_m}
\end{corollary}

\begin{proof}
We claim that if $[F]\in \hat{T}^d$ (as defined in Section \ref{inc_corr_subsection}) then $V(F)_{\sing}$ contains an integral $b$-dimensional
subscheme of degree $d$. Indeed, $V(F)_{\sing}$ contains some integral $b$-dimensional closed subscheme of degree
$\tilde{d}\in\{1,...,d\};$
if $[F]\notin\cup_{d'=1}^{d-1}T^{d'}_k,$ then necessarily $\tilde{d}=d$.

Now, we can induct on $d$, so assume that $Z\nsubseteq \cup_{d'=1}^{d-1}T^{d'}_k.$ Note that
 $Z-\left(Z\cap (\cup_{d'=1}^{d-1} T^{d'})\right)\subset Z$
is a dense open subset of $Z$, which therefore has the same dimension as $Z$, but is contained in
$\hat{T}^d\subset \rho(R^d)\subset\rho(\overline{R^d}).$
Thus $\dim Z\leq\dim\rho(\overline{R^d})<\dim X^1,$ by Lemma \ref{sml}
\end{proof}

\subsection{Preparations for the computation of the second largest component}
\label{seccompsmldegsection}

Here we discuss a calculation similar to the one in the previous section, which will later be used for the computation of the
dimension of the second largest component of $X$. Define
\[\beta_2(l)=\binom{l+b+1}{b+1}-\binom{l+b-3}{b+1}+(n-b-1)\left(\binom{l+b}{b+1}-\binom{l+b-2}{b+1}\right)\]
and set $\gamma_2(l)=\beta_2(l)+1-(b+2)n+\frac{b(b+1)}{2}.$ We will later see that $\binom{l+n}{n}-\gamma_2(l)$ is the dimension
of the second largest component of $X$, at least when $\ch k\neq 0.$

\begin{lemma}
Fix an integer $B$.
There exists $l_0$ (effectively computable) such that for all pairs $(d,l)$ with
$3\leq d\leq B$ and $l\geq l_0$ (if $b=n-1,$ assume $d\geq 4$),
and any irreducible component $Z$ of $T^d_k,$ either $Z\subset T^1_k\cup T^2_k,$ or
\[\dim Z<\binom{l+n}{n}-\gamma_2(l).\]
\label{scompsmldeg}
\end{lemma}
(In the case $b=n-1, d=3$, we will prove a slightly weaker but sufficient statement in
Remark \ref{techdetscomprem}.)

\begin{proof}
Precisely as in Lemma \ref{sml}, because of inequality (\ref{dimfiberineq}), it suffices to establish the inequality
\[\dim\P(W_{\cup L_i})+\dim\widetilde{\Hilb^d}<\binom{l+n}{n}-\gamma_2(l),\qquad \text{i.e.,}\]
\[\gamma_2(l)-1+\dim\widetilde{\Hilb}^d<\codim_V(W_{\cup L_i}).\]
Let $c=-(b+2)n+\frac{b(b+1)}{2}$.
By Lemma \ref{codi_d_spaces}, it suffices to prove that
\[\gamma_2(l)-1+\dim \widetilde{\Hilb}^d <\binom{l+b}{b}+(n-b)\sum_{e=1}^d\binom{l-2e+1+b}{b},\]
or, equivalently, that
\[\binom{l+b-3}{b}+(n-b)\binom{l+b-3}{b-1}+c+\dim\widetilde{\Hilb^d}<(n-b)\sum_{e=3}^d\binom{l-2e+1+b}{b}\]
for the appropriate values of $d$ and $l$. For a fixed $d\in\{3,...,B\},$ the right hand side is a polynomial in $l$ of degree $b$ and leading coefficient
$\frac{(n-b)(d-2)}{b!}$, while the left hand side is a polynomial in $l$ of degree $b$ but smaller leading coefficient $\frac{1}{b!}$.
\end{proof}

\begin{remark}
Again, using the Eisenbud--Harris bound for the dimension of $\widetilde{\Hilb^d}$ (conjectural for $b\geq 2$),
we can prove that there exists an effectively computable $l_0$ so that the conclusion of the above Lemma holds
for all pairs $(d,l)$ with $3\leq d\leq\frac{l+1}{2}$ and $l\geq l_0$. Again, this approach would produce a smaller value for $l_0$ but would be conditional when $b\geq 2$.
\end{remark}

\section{The case of large degree $d$, when $k=\overline{\FF_p}$}
\label{lrg_d_p}

Fix $n$ and $1\leq b\leq n-1$ as usual, and fix a prime $p$. Let $k=\overline{\FF_p}.$ Recall the definition of $A_b(l,m)$ from
Section \ref{specCsubVF} and the definition of $T^d_k$ and $\hat{T}^d$ from Section \ref{inc_corr_subsection}.

The goal of this section is to handle the case of large $d$ when $k=\overline{\FF_p}.$ Specifically, we prove the following

\begin{proposition}
Fix a triple of positive integers $(l,m,a)$. Set
$\tau=\lfloor\frac{l-1}{p}\rfloor$ and  $m'=\min(m,\tau+1).$ Suppose that
\[\binom{\tau+b+1}{b+1}>a-1\quad\text{and}\quad A_b(\tau,m')>a-1.\]
Let $d\geq m.$ If $Z$ is an irreducible component of $T^d_k$, then either $Z\subset T^{d'}_k$ for some $1\leq d'<d,$ or
\[\dim Z\leq \binom{l+n}{n}-a.\]
\label{lrg_m}
\end{proposition}

Let $Z\subset T^d_k$ be an irreducible component (notation and assumptions as above). Suppose that
$Z\nsubseteq \bigcup_{d'=1}^{d-1}T^{d'}_k.$ Then $Z-\left(Z\cap (\bigcup_{d'=1}^{d-1} T^{d'}_k)\right)\subset Z$ is a dense
open subset,  and
therefore is of the same dimension as $Z$. It is contained in $\hat{T}^d$.
So the goal is now to prove that $\dim \hat{T}^d\leq \binom{l+n}{n}-a$.

\subsection{Reduction to a problem over finite fields}
\label{redffsubsect}

We begin with a general discussion, which applies to any (quasiprojective) variety over $\overline{\FF_p}.$
Let $T=\cap V(G_i)-\cap V(G_j')\subset \P^M_{\overline{\FF_p}}$ be a quasiprojective variety over $\overline{\FF_p}$,
where $G_i,G_j'\in\overline{\FF_p}[y_0,...,y_M]$.
Let $A$ be an integer, and suppose we want to prove that $\dim T\leq A$.
There is a finite field $\FF_{q_0}$ such that $G_i,G_j'\in\FF_{q_0}[y_0,...,y_M],$ so $T$ comes from
$T_0:=\cap V(G_i)-\cap V(G_j')\subset\P^M_{\FF_{q_0}},$ which is now a variety over $\FF_{q_0}.$
We know that $\dim T=\dim T_0,$ so suffices to prove that $\dim T_0\leq A$.
For this, by the result of Lang-Weil \cite{LW}, it suffices to prove that
$\#T_0(\FF_q)=O(q^A)$ as $q\to\infty$ (through powers of $q_0$ of course).

Consider now $T=\hat{T}^d\subset \P(\overline{\FF_p}[x_0,...,x_n]_l)$, and let $\hat{T}^d_0$ (a variety over a finite field
$\FF_{q_0}$) be as in the previous paragraph.
In particular, $\hat{T}^d_0(\FF_q)$ consists of all $[F]\in (\FF_q[x_0,...,x_n]_l-\{0\})/\FF_q^*$ such that when we regard $[F]$ in
$(\overline{\FF_p}[x_0,...,x_n]_l-\{0\})/\overline{\FF_p}^*$, we have that
$[F]\in\hat{T}^d\subset \P(\overline{\FF_p}[x_0,...,x_n]_l)$.

\begin{remark}
Even if $F$ has coefficients in $\FF_q,$ we always consider
$V(F)$ and $V(F)_{\sing}$ as subschemes of $\P^n_{\overline{\FF_p}}$ by first
regarding $F$ in $\overline{\FF_p}[x_0,...,x_n].$
\end{remark}

By the argument in the proof of Corollay \ref{sml_m},
the set $\hat{T}^d_0(\FF_q)$ is a subset of
\begin{multline*}
\widetilde{T}^d:=\{ [F]\in(\FF_q[x_0,...,x_n]_l-\{0\})/\FF_q^*\mid
 V(F)_{\sing}\subset\P^n_{\overline{\FF_p}}\ \text{contains}\\ \text{an integral $b$-dimensional subscheme
(over $\overline{\FF_p}$) of degree}\ d\}.
\end{multline*}

So  our goal now is to prove that $\#\widetilde{T}^d=O(q^{\binom{l+n}{n}-a})$ as $q\to\infty$ (through powers of $q_0$).

As $F$ is chosen randomly from $\FF_q[x_0,...,x_n]_l,$ let $\Lambda$ be the event that $V(F)_{\sing}$
contains an {\it integral} $b$-dimensional subscheme of degree $d.$
Thus, our task is to prove that
$\Prob(\Lambda)q^{\binom{l+n}{n}}=O( q^{\binom{l+n}{n}-a+1}),$ or equivalently,
that $\Prob(\Lambda)=O(q^{-a+1})$ as $q\to\infty$ (through powers of $q_0$).

\subsection{Final preparations}

Consider the natural homogenization map $\sim\colon \FF_q[x_0,...,x_{n-1}]_{\leq l}\xto{\simeq}\FF_q[x_0,...,x_n]_l$
with respect to the variable $x_n.$
We have to be slightly careful because this is not the usual homogenization map (which takes a polynomial and
homogenizes it to the smallest possible degree); we think of $\sim$ as ``homogenization-to-degree-$l$" map. Recall that
$\tau=\lfloor \frac{l-1}{p}\rfloor.$

\begin{lemma} Let $Z\subset \P^n_{\overline{\FF_p}}$ be an integral closed subscheme not contained in the
hyperplane $V(x_n).$ Let
$F_0\in \FF_q[x_0,...,x_{n-1}]_{\leq l-1}$ be a fixed polynomial. Then, as $G$ is chosen randomly from
$\FF_q[x_0,...,x_{n-1}]_{\leq\tau},$ we have
\[\Prob(Z\subset V((F_0+G^p)^\sim)) \leq \Prob(Z\subset V(G^\sim)).\]
Here, the first $\sim$ is homogenization to degree $l-1,$ and the second one is homogenization
to degree $\tau$.
\label{dhmg}
\end{lemma}

\begin{proof}
Let $I\subset \overline{\FF_p}[x_0,...,x_{n-1}]$ be the (radical) ideal of
$Z\cap D_+(x_n)\subset D_+(x_n).$ We claim that for an inhomogeneous polynomial $H\in\FF_q[x_0,...,x_{n-1}]_{\leq l-1},$ we
have $Z\subset V(H^\sim)$ if and only if $H\in I.$ For this, first notice that $V(H^\sim)$
 is either $V(H)^-$ or $V(H)^-\cup V(x_n)$
(where $V(H)^-$ is the topological closure of $V(H)\subset D_+(x_n)$ in $\P^n_{\overline{\FF_p}}$),
depending on whether or not the degree of $H$ is equal to the degree of homogenization of the map $\sim$. Since $Z$ is irreducible
and not contained in $V(x_n),$ we have
$Z\subset V(H^\sim)$ if and only if $Z\subset V(H)^-$.
 In turn, since $Z\cap D_+(x_n)\neq\emptyset,$ this condition is
equivalent to $Z\cap D_+(x_n)\subset V(H),$ which is precisely the condition $H\in I.$

Therefore, $Z\subset V((F_0+G^p)^\sim)$ if and only if $F_0+G^p\in I.$
 If $F_0+G^p\in I$ and $F_0+G_1^p\in I,$ then
$(G-G_1)^p\in I,$ and hence $G':=G-G_1\in I$. So the number
of $G$ with $F_0+G^p\in I$ is either zero, or is equal to the number of elements $G'\in I$ with
$G'\in \FF_q[x_0,...,x_{n-1}]_{\leq\tau}.$ This is precisely the number of $G'\in\FF_q[x_0,...,x_{n-1}]_{\leq\tau}$ such that
$Z\subset V((G')^\sim).$
\end{proof}

\begin{corollary}
Keep the notation of Lemma \ref{dhmg}.
\begin{itemize}
\item[a)] If $\dim Z\geq b+1,$ then
\[\Prob(Z\subset V((F_0+G^p)^\sim))\leq q^{-\binom{\tau+b+1}{b+1}}.\]
\item[b)] If $\dim Z=b$ and $\deg Z=d\geq m,$ then
\[\Prob(Z\subset V((F_0+G^p)^\sim))\leq q^{-A_b(\tau,m')},\]
where $m'=\min(m,\tau+1).$
\end{itemize}
\label{dhcor}
\end{corollary}

\begin{proof}
 Combine Lemma \ref{dhmg} with Corollaries \ref{spdmprob} and \ref{spSprob}.
\end{proof}

\subsection{The key step (large degree $d$)}

Fix a triple $(l,m,a)$ of positive integers. Recall that $\tau=\lfloor\frac{l-1}{p}\rfloor$ and $m'=\min(m,\tau+1).$ Let $d\geq m.$

As $F^\sim$ is chosen randomly from $\FF_q[x_0,...,x_n]_l$, or, equivalently, as $F$ is chosen randomly from
 $\FF_q[x_0,...,x_{n-1}]_{\leq l},$ let $E_n$ be the event that the following two conditions are satisfied:
\begin{itemize}
\item For each $i=0,...,n-b-1,$ the variety $V(\frac{\d F^\sim}{\d x_0},...,\frac{\d F^\sim}{\d x_i})$
has all irreducible components of dimension
$n-i-1,$ except possibly for components contained in the hyperplane $V(x_n).$
\item If $C\subset V(\frac{\partial F^\sim}{\partial x_0},...,\frac{\partial F^\sim}{\partial x_{n-b-1}})$ is a
$b$-dimensional integral closed subscheme of degree $d,$
then either $C\subset V(x_n)$, or
$C\nsubseteq V(\frac{\partial F^\sim}{\partial x_{n-1}})$.
\end{itemize}

We now proceed to bound $\Prob(E_n)$ from below (this is the hard part).

\begin{lemma}
\begin{equation}
\Prob(E_n)\geq
\left(\prod_{i=0}^{n-b-1}\left(
1-\frac{(l-1)^i}{q^{\binom{\tau+b+1}{b+1}}}
\right)\right)
\left(
1-\frac{(l-1)^{n-b}}{q^{A_{b}(\tau,m')}}
\right).
\label{bigpro}
\end{equation}
\label{probcount}
\end{lemma}

\begin{proof}
We now mimic the main argument in \cite[Section 2.3]{P}. We will generate a random $F$ by choosing
$F_0\in \FF_q[x_0,...,x_{n-1}]_{\leq l}, G_i\in\FF_q[x_0,...,x_{n-1}]_{\leq\tau}$ randomly, in turn, and then setting
\begin{equation}
F:=F_0+G_0^px_0+\dots+G_{n-1}^p x_{n-1}.
\label{smart}
\end{equation}
For $F\in \FF_q[x_0,...,x_{n-1}]_{\leq l-1},$ the number of tuples $(F_0,G_0,...,G_{n-1})$ for which
(\ref{smart}) holds is independent of $F$.
We have
\[\frac{\d F}{\d x_i}=\frac{\d F_0}{\d x_i}+G_i^p.\] Moreover, the homogenization map $\sim$ commutes with differentiation,
so
\[\frac{\partial F^\sim}{\partial x_i}=\left(\frac{\partial F_0}{\partial x_i}+G_i^p\right)^\sim\]
(again, the two uses of $\sim$ here refer to homogenizations to different degrees, $l$ and $l-1$, respectively).

Let $i\in\{0,...,n-b-1\}.$ Suppose that $F_0,G_0,...,G_{i-1}$ are fixed such that
 $V(\frac{\d F^\sim}{\d x_0},...,\frac{\d F^\sim}{\d x_{i-1}})$ has only $(n-i)$-dimensional components,
 except possibly for components contained in the hyperplane $V(x_n).$
By B\'ezout's theorem (see p.~10 in \cite{Ful} for the version we are using here),
$V(\frac{\d F^\sim}{\d x_0},...,\frac{\d F^\sim}{\d x_{i-1}})$ has at most $(l-1)^i$ irreducible components.
Let $Z$ be one of them, and suppose that $Z\nsubseteq V(x_n).$ As $G_i$ is chosen randomly from $\FF_q[x_0,...,x_{n-1}]_{\leq\tau},$
we claim that
\[\Prob\left(Z\subset V\left(\frac{\d F^\sim}{\d x_i}\right)\right)\leq q^{-\binom{\tau+b+1}{b+1}}.\]
This follows from Corollary \ref{dhcor}a, since $\dim Z=n-i\geq b+1$.

For the final step,
conditioned on a choice of $F_0,G_0,...,G_{n-b-1}$
such that $V(\frac{\d F^\sim}{\d x_0},...,\frac{\d F^\sim}{\d x_{n-b-1}})$
has only $b$-dimensional components,
 except possibly for components contained in $V(x_n),$ we claim that
the probability, as $G_{n-1}\in\FF_q[x_0,...,x_{n-1}]_{\leq\tau}$,
that some $b$-dimensional component $C$ of $V(\frac{\d F^\sim}{\d x_0},...,\frac{\d F^\sim}{\d x_{n-b-1}})$
of degree $d$ and not contained in $V(x_n)$, is contained in
$V(\frac{\d F^\sim}{\d x_{n-1}})$, is at most $(l-1)^{n-b} q^{-A_{b}(\tau,m')}$.

Indeed,
the number of $b$-dimensional components $C$ of $V(\frac{\d F^\sim}{\d x_0},...,\frac{\d F^\sim}{\d x_{n-b-1}})$ of degree
$d$ is at most
$(l-1)^{n-b},$ by B\'ezout's theorem again (this is a bound on the total number of components of all dimensions).
If we fix a $b$-dimensional component $C$ of degree $d$ and not contained in $V(x_n)$, for
fixed $F_0,G_0,...,G_{n-b-1}$,
the probability (as $G_{n-1}$ is chosen randomly from $\FF_q[x_0,...,x_{n-1}]_{\leq\tau}$)
that $C\subset V\left((\frac{\d F_0}{\d x_{n-1}}+G_{n-1}^p)^\sim\right)$, is at most
$q^{-A_{b}(\tau,m')},$ by Corollary \ref{dhcor}b.
\end{proof}

\begin{proof}[Proof of Proposition \ref{lrg_m}]
By the hypothesis of Proposition \ref{lrg_m}, each of the exponents on the right hand side of (\ref{bigpro}) is greater than
$a-1.$
By virtue of the inequality $\prod (1-\e_i)\geq 1-\sum \e_i,$
Lemma \ref{probcount} implies that $\Prob(E_n)\geq 1-\frac{1}{q^{a-1}}$ for large $q$.
Therefore,
\[1-\Prob(E_n)=O\left(\frac{1}{q^{a-1}}\right)\quad\text{as}\ q\to \infty.\]

As $F\in\FF_q[x_0,...,x_{n-1}]_{\leq l},$
let $E_n'$ be the event that any integral $b$-dimensional closed subscheme $C\subset V(F)_{\sing}$ of degree
$d$ is contained in $V(x_n).$ Then $E_n$ implies $E_n'.$
For each $i=0,...,n-1,$ define $E_i,E_i'$ in analogy with $E_n,E_n'$,
except with dehomogenization with respect to the variable $x_i$
(and any ordering of the remaining variables).
The same conclusion $1-\Prob(E_i)=O(\frac{1}{q^{a-1}})$ holds for all $i=0,...,n.$ Note that
 $\Lambda$ (defined at the end of Section \ref{redffsubsect})
 implies $\bigcup_{i=0}^n \overline{E_i'}$, where  $\overline{E_i'}$ denotes the event opposite to $E_i'$.
 Indeed, $\bigcap V(x_i)=\emptyset$, so we cannot have $C\subset V(F)_{\sing}$ contained
 in all the coordinate hyperplanes.
 Therefore,
 \[\Prob(\Lambda)\leq\sum_{i=0}^n (1-\Prob(E_i'))\leq\sum_{i=0}^n (1-\Prob(E_i))=O\left(\frac{1}{q^{a-1}}\right)\
 \text{as}\ q\to\infty,\]
as desired.
\end{proof}

\section{Proof of the main theorem}

We now put together the main results Corollary \ref{sml_m} and Proposition \ref{lrg_m}
and finish the proof of Theorem \ref{mainthm}. Namely, Theorem \ref{mainthm}
follows immediately from our previous work when
$k=\overline{\FF_p},$ and we use upper-semicontinuity applied to $T^d\to\Spec\Z$ to prove the case $\ch k=0$
(Section \ref{eofpfsubsection}). However, there are technicalities (Corollary \ref{techforuniincharzero})
concerning the uniqueness of the largest component in
characteristic $0$, which we discuss in Section \ref{uniqlrgcompsubsection}. Finally,
we prove Theorem \ref{corscomp} in Section \ref{seclrgcompsubsectioneopf}.

\subsection{Restatement of the problem and the end of the proof}
\label{eofpfsubsection}

\begin{lemma}
Let $[F]\in \P(V)$ be such that $\dim V(F)_{\sing}\geq b.$ Then there is an integral $b$-dimensional closed subscheme
$C\into\P^n$ of degree at most
$l(l-1)^{n+1}$ such that $C\subset V(F)_{\sing}.$
\label{Bez}
\end{lemma}

\begin{proof}
Let $Z_1,...,Z_s$ be the irreducible components of
$V(F)_{\sing}=V(F,\frac{\partial F}{\partial x_0},...,\frac{\partial F}{\partial x_n})$. Then by B\'ezout's theorem
(\cite{Ful}, p.~10),
\[\sum_{i=1}^s \deg(Z_i)\leq \deg(F)\prod_{\substack{ 0\leq j\leq n\\ \d F/\d x_j\neq 0}}
\deg\left(\frac{\partial F}{\partial x_j}\right)\leq l(l-1)^{n+1}.\]
But some component $Z_i$ has dimension at least $b$, so, intersecting with hyperplanes if necessary,
 this component will contain an integral $b$-dimensional closed subscheme of degree at most
$\deg(Z_i)\leq l(l-1)^{n+1}.$
\end{proof}

\begin{proof}[Proof of Theorem \ref{mainthm}]
By Lemma \ref{Bez}, $X$ is a finite union:
\begin{equation}
X=\bigcup_{d=1}^{l(l-1)^{n+1}}T^d_k.
\label{decforx}
\end{equation}
In particular, $X$ is a closed subset of $\P(V)$.
The statement of Theorem \ref{mainthm} is now equivalent to the following one:
for any $d\geq 2,$ we have $\dim (T^d_k-T^1_k)<\dim X^1$.
But $T^d_k-T^1_k=(T^d-T^1)_k,$ and if
$k_0\subset k$ is a subfield, then
$\dim (T^d-T^1)_k=\dim (T^d-T^1)_{k_0}.$ So it suffices to assume that $k=\overline{\FF_p}$ or $k=\overline{\Q}$.

First, suppose that $k=\overline{\FF_p}.$
Set
$B=p^b(n-b+1)$ in Lemma \ref{sml}, and set
$m=B+1$ in Proposition \ref{lrg_m}. Let $\tau(l)=\lfloor\frac{l-1}{p}\rfloor$.
By the definition in Lemma \ref{cont_C_lrg_deg} and the definition of $a_{n,b}(l)$ from the introduction, we have that
$A_b(\tau,m)$ grows as a polynomial in $l$ of degree $b$ and leading coefficient
$\frac{m}{p^b b!}>\frac{n-b+1}{b!},$ so
$A_b(\tau,m)>a_{n,b}(l)$ for sufficiently large $l$.
Also, $\binom{\tau(l)+b+1}{b+1}>a_{n,b}(l)$ for $l\gg 0$.
Thus, there is an effectively computable $l_0$ which satisfies Corollary \ref{sml_m} and such that
the hypothesis of Proposition \ref{lrg_m} is satisfied for the triple $(l,B+1,a_{n,b}(l)+1)$
whenever $l\geq l_0$.

We now prove by induction on $d\geq 2$ that for
any irreducible component $Z$ of $T^d_{k},$ either
$Z=X^1$ or $\dim Z<\dim X^1$. For $2\leq d\leq B$ this follows from Corollary \ref{sml_m}. Let $d\geq B+1$.
Assume that the statement holds
for all $2\leq d'\leq d-1.$ Then it also holds for $d$, by Proposition \ref{lrg_m}.

Now, let $k=\overline{\Q}.$ Let $p$ be any prime, and consider $l\geq l_0(n,b,p)$ as above.
By the previous paragraph,
for any $d\geq 2,$
$\dim T^d_{\FF_p}=\dim T^d_{\overline{\FF_p}}\leq \dim X^1.$
But, since $T^d\to\Spec\Z$ is projective, by the upper semicontinuity theorem, we know
\[\dim T^d_{\overline{\Q}}=\dim T^d_{\Q}\leq \dim T^d_{\FF_p}\leq \dim X^1.\]
Therefore, as long as $l\geq l_0(n,b,p)$ for some $p$ (take $p=2$ to obtain the best value of $l_0$ here),
we know that $X^1$ (over $\overline{\Q}$) is {\it an} irreducible component of
$X$ (over $\overline{\Q}$) of maximal dimension.

We now address the question of uniqueness of $X^1$ as a largest component.
In Section \ref{uniqlrgcompsubsection} we will show that it is possible to choose $p$ such that
$X^1\nsubseteq T^d_{\overline{\FF_p}}$ for any $d\geq 2.$
For such $p$, and for $d\geq 2,$ the conclusion from two paragraphs ago implies
\[\dim T^d_{\overline{\FF_p}}<\dim X^1.\]
So
\[\dim T^d_{\overline{\Q}}\leq\dim T^d_{\overline{\FF_p}}<\dim X^1.\]
By (\ref{decforx}), any irreducible component of $X$ is either $X^1$ or is contained in $T^d_k$ for some $d\geq 2.$
This completes the proof.
\end{proof}

\begin{remark}
We postpone for the next section the fact that over $\overline{\FF_p},$ we have $X^1\nsubseteq T^d_{\overline{\FF_p}},$
provided that $p\neq 2$ or $n-b$ is even.
So for $l\geq l_0(n,b,2),$ we know that $X^1$ is an irreducible component of $X$ of largest dimension; for $l\geq l_0(n,b,2)$
when $n-b$ is even, and for $l\geq l_0(n,b,3)$ when $n-b$ is odd,
we also know that $X^1$ is the unique largest-dimensional component of $X$.
\end{remark}

\begin{remark}
If we used the Eisenbud--Harris bounds for the dimension of the restricted Hilbert scheme, we would obtain a better bound for $l_0$ as follows. In the above proof, when $k=\overline{\FF_p},$
let $\tau(l)=\lfloor\frac{l-1}{p}\rfloor$ and $m(l)=\lceil\frac{l+1}{2}\rceil$.
There exists an effectively computable $l_0$ which satisfies Remark \ref{sml_better_bound} and such that
Proposition \ref{lrg_m} applies to the triple $(l,m(l),a_{n,b}(l)+1)$ whenever $l\geq l_0$.
Then for any $l\geq l_0(n,b,p),$ the statement of Theorem \ref{mainthm} holds.
\end{remark}

\subsection{Uniqueness of the largest component (in characteristic $0$)}
\label{uniqlrgcompsubsection}

We set the following notation for this section.
Consider a $b$-dimensional closed subscheme $C=V(f,x_{b+2},...,x_n)$ of $\P^n$,
where $f\in k[x_0,...,x_{b+1}]_d-\{0\},$
and set $W=(f,x_{b+2},...,x_n)^2_l$. In order to finish the proof of Theorem \ref{mainthm}, it will be sufficient to consider the
case when $C$ is a linear $b$-dimensional subspace in the next lemma; however, we will use the more general statement (when $d=2$)
in Section \ref{seclrgcompsubsectioneopf}.

\begin{lemma}
Assume $l\geq 2d+1.$
There is a dense open subset $U_1\subset\P(W)$ such that for all $[F]\in U_1,$ $V(F)_{\sing}=C$ (set-theoretically).
\label{ctrlsng}
\end{lemma}

\begin{proof}
Consider the incidence correspondence
\[Y_1=\{([F],P)\in \P(W)\times (\P^n-C)\mid P\in V(F)_{\sing}\}\subset \P(W)\times (\P^n-C).\]
We are going to show that $\dim Y_1<\dim\P(W)$; this will imply that the closure $\overline{Y_1}$ of $Y_1$ in $\P(W)\times\P^n$
also has dimension smaller
than that of $\P(W),$ and thus the image of this closure under the projection to $\P(W)$ will be a proper closed subset of $\P(W).$
 Its complement $U_1$ will satisfy the condition of the lemma.

Consider the second projection $\tau\colon Y_1\to\P^n-C,$ and let $P\in\P^n-C.$
We claim the fiber $\tau^{-1}(P)$ is a projective linear subspace of $\P(W)$ of codimension $n+1$. This will imply that $Y_1$ is
irreducible, of dimension $\dim Y_1=\dim\P(W)-1.$

Suppose first that $P\in\cup_{i=b+2}^n D_+(x_i).$ Without loss of generality, assume that $P=[a_0,...,a_{n-1},1].$
Notice that $\tau^{-1}(P)$ is just
\[\P\left(((x_0-a_0 x_n,...,x_{n-1}-a_{n-1}x_n)^2\cap (f,x_{b+2},...,x_n)^2)_l\right)\subset\P(W),\]
and it is easy to show that
\[\dim \left(\frac{W}{(x_0-a_0 x_n,...,x_{n-1}-a_{n-1}x_n)^2\cap (f,x_{b+2},...,x_n)^2}\right)_l=n+1.\]

Suppose now that $P\in V(x_{b+2},...,x_n),$ without loss of generality $P=[1,a_1,...,a_{b+1},0,...,0].$
We have to prove that the following map
is an isomorphism:
\begin{align*}
&\left(\frac{(f,x_{b+2},...,x_n)^2}{(x_1-a_1 x_0,...,x_{b+1}-a_{b+1}x_0,x_{b+2},...,x_n)^2\cap(f,x_{b+2},...,x_n)^2}\right)_l\into\\
&\left(\frac{S}{(x_1-a_1 x_0,...,x_{b+1}-a_{b+1}x_0,x_{b+2},...,x_n)^2}\right)_l\simeq\\
&k[x_0]_l\oplus\left(\bigoplus_{i=1}^{b+1}k[x_0]_{l-1}(x_i-a_i x_0)\right)\oplus\left(\bigoplus_{i=b+2}^n k[x_0]_{l-1}x_i\right).
\end{align*}

Dehomogenize $f$ with respect to $x_0,$ consider a Taylor expansion at $(a_1,...,a_{b+1}),$ and homogenize to degree $l$ again,
so $f\equiv ax_0^d\pmod{(x_1-a_1x_0,...,x_{b+1}-a_{b+1}x_0)}$ with $a\neq 0$. So
$f^2\equiv a^2x_0^{2d}\pmod{(x_1-a_1x_0,...,x_{b+1}-a_{b+1}x_0)}.$
Now, the elements $f^2x_0^{l-2d-1}(x_i-a_ix_0)$ (for $i=1,...,b+1$), $f^2x_0^{l-2d-1}x_i$ (for $i=b+2,...,n$), and
$f^2x_0^{l-2d}$ map to a basis of the target.
\end{proof}

 We will use the lemma below only when $C$ is linear,
but we prove it here for a more general $C$ for the purposes of the later discussion in Remark \ref{techcompl}.

\begin{lemma}
Suppose that $l\geq 2d.$
If $\text{char } k\neq 2,$ then
there exists a dense open subset $U_2\subset\P(W)$ such that for all $[F]\in U_2,$ we have
\[\dim\{P\in C\mid\dim T_P V(F)_{\sing}\geq b+1\}\leq b-1.\]
If $\text{char }k=2$ and $C$ is a $b$-dimensional linear subspace,
and $n-b$ is even, then the same conclusion holds.
\label{ctrlsngred}
\end{lemma}

\begin{proof}
Consider the incidence correspondence
\[Y_2=\{([F],P)\in\P(W)\times C\mid\dim T_P V(F)_{\sing}\geq b+1\}\subset \P(W)\times C\]
(this is a closed subset). We will show that
$Y_2\neq \P(W)\times C,$ i.e.,
$\dim Y_2\leq\dim\P(W)+b-1$. Once this is done,
the map
$Y_2\to\P(W)$ will give a dense open $U_2\subset\P(W)$ such that the fiber over any $[F]\in U_2$ has
dimension at most $b-1$.

Suppose that $\ch k\neq 2.$
Fix a point $P=[p_0,...,p_{b+1},0,...,0]\in C$
with at least $2$ nonzero coordinates
such that $V(f)\subset \P^{b+1}=V(x_{b+2},...,x_n)$ is smooth at $P$.
Without loss of generality, $\frac{\d f}{\d x_{b+1}}(P)\neq 0$ and $p_0\neq 0.$
We claim that there exists $[F]\in \P(W)$ with $\dim T_P V(F)_{\sing}\leq b$.

For $[F]\in \P(W),$ we have $V(F)_{\sing}=V(F,\frac{\d F}{\d x_0},...,\frac{\d F}{\d x_n})$, so we have to look at the
Jacobian
\[J(P)=\begin{pmatrix} \frac{\d F}{\d x_0}(P) & \frac{\d F}{\d x_1}(P) & \dots &\frac{\d F}{\d x_n}(P)\\
\frac{\d^2 F}{\d x_0^2}(P) &\frac{\d^2 F}{\d x_0\d x_1}(P) & \dots & \frac{\d ^2 F}{\d x_0\d x_n}(P)\\
\vdots &\vdots &\ddots&\vdots\\
\frac{\d^2 F}{\d x_n\d x_0}(P) &\frac{\d^2 F}{\d x_n\d x_1}(P) & \dots & \frac{\d ^2 F}{\d x_n^2}(P)
\end{pmatrix}.\]
We know that $\dim T_P V(F)_{\sing}=n-\text{rk}J(P),$ so $\dim T_P V(F)_{\sing}\leq b$ if and only if $\text{rk}J(P)\geq n-b.$
In other words, we have to give some $[F]\in\P(W)$ such that some $(n-b)\times(n-b)$ minor of the Jacobian is nonzero.
Consider
\[F=x_0^{l-2d}f^2+\sum_{i=b+2}^n x_0^{l-2}x_i^2\]
and note that the bottom right $(n-b)\times (n-b)$ minor of $J(P)$ is nonzero.

Now suppose that $\text{char }k=2$ but $n-b$ is even and $C=V(x_{b+1},...,x_n)$. Let $P=[1,0,...,0]$. Consider
$F=\sum_{i=1}^{\frac{n-b}{2}} x_{b+2i-1}x_{b+2i}x_0^{l-2}.$
\end{proof}

\begin{remark}
This lemma fails when $\text{char }k=2$, $C$ is linear, and $n-b$ is odd.
\end{remark}

\begin{corollary}
Suppose that $\text{char }k\neq 2$ or $\text{char }k=2$ but $n-b$ is even.
 Then $X^1\nsubseteq T^d_{\overline{\FF_p}}$ for any $d\geq 2.$
\label{techforuniincharzero}
\end{corollary}

\begin{proof}
Let $C=V(x_{b+1},...,x_n).$ Let $U_1$ and $U_2$ be as given by Lemmas \ref{ctrlsng} and \ref{ctrlsngred}. Let $U=U_1\cap U_2.$
So $U$ is a dense open subset of $P(W)$ such that for all $[F]\in U,$ $V(F)_{\sing}=L$ set-theoretically,
and in addition, the closed embedding
$L\into V(F)_{\sing}$ is an isomorphism over the complement of a closed subset of smaller dimension. Thus the Hilbert polynomial
of $V(F)_{\sing}$ has degree $b$ and leading term
$1/b!$, so $V(F)_{\sing}$ does not contain any closed subscheme of dimension $b$ and degree $d\geq 2.$
In other words, $[F]\in X^1-T^d_{\overline{\FF_p}}.$
\end{proof}

Similarly, we can apply Lemmas \ref{ctrlsng} and \ref{ctrlsngred} to an integral $C=V(f,x_{b+2},...,x_n)$ of degree $2$
and obtain the following

\begin{corollary} Suppose that $\ch k\neq 2.$
There exists $[F]\in\P(V)$ such that $V(F)_{\sing}$ is a $b$-dimensional integral closed subscheme of degree $2$
(as a set), and such that
$V(F)_{\sing}$ does not contain any $b$-dimensional closed subscheme of degree $d\geq 3.$
\label{ingforuniqseccomp}
\end{corollary}

\subsection{The second largest component}
\label{seclrgcompsubsectioneopf}

In contrast to the treatment of the largest component of $X$,
the existence of a component of the expected second-largest dimension is a little more subtle,
so there will be an extra twist in the argument. In this section, we prove Theorem \ref{corscomp}.

For now, $k$ is again any algebraically closed field.

Fix $n,b$ as usual, and let $d\geq 1$. Define
\begin{align*}\beta_d(l)
&=\binom{l+b+1}{b+1}-\binom{l-2d+b+1}{b+1}+(n-b-1)\left(\binom{l+b}{b+1}-\binom{l-d+b}{b+1}\right)\\
&=\frac{(n-b+1)d}{b!}l^b+\dots
\end{align*}
and recall the definition of $\gamma_2(l)$ from Section \ref{seccompsmldegsection}.

Let $I=(f,x_{b+2},...,x_n)\subset S=k[x_0,...,x_n],$ where $f\in k[x_0,...,x_{b+1}]_d-\{0\}$.
Consider the composition
\[\Phi\colon  k[x_0,...,x_{b+1}]_l\oplus\left(\bigoplus_{i=b+2}^n k[x_0,...,x_{b+1}]_{l-1}x_i\right)
\into S_l\onto S_l/(I^2\cap S_l).\]
Note that $\Phi$ is surjective.

\begin{lemma}
 We have that
\[\ker(\Phi)=\{P+\sum_{i=b+2}^n P_i x_i\ \colon\ f^2|P, f|P_i\ \text{for}\ i=b+2,...,n\}.\]
For $l\geq 2d$, the codimension of $I^2_l$ in $S_l$ equals $\beta_d(l)$.
\label{explct}
\end{lemma}

\begin{proof} If $P+\sum P_i x_i\in\ker(\Phi),$ then we can write $P+\sum P_i x_i=T\in I^2.$ Expand both sides
as polynomials in $x_{b+2},...,x_n$ and just compare the two expressions.
The second part is an immediate consequence.
\end{proof}

\begin{lemma} Let $C\into\P^n$ be any integral $b$-dimensional closed subscheme of degree $2$,
with (saturated) ideal $I$. If $F\in k[x_0,...,x_n]_l$ satisfies $C\subset V(F)_{\sing},$ then $F\in I^2_l.$
\label{sq_of_ideal}
\end{lemma}

\begin{proof}
We can assume
that $C=V(I),$ with $I=(f,x_{b+2},...,x_n),$ where $f\in k[x_0,...,x_{b+1}]_2-\{0\}$ is irreducible.
We claim that the ideal $I^2$ is saturated. Indeed, let $F\in S$ be homogeneous, and suppose that $x_j^M F\in I^2$ for all
$j=0,...,n$ (and for some $M$).
Write $F=P+\sum_{i=b+2}^n P_i x_i+T,$ where $P,P_i\in k[x_0,...,x_{b+1}]$ are homogeneous of the appropriate degrees, and
$T\in (x_{b+2},...,x_n)^2.$ Since $x_0^M F\in I^2,$ Lemma \ref{explct} implies that $f^2|x_0^M P$ and $f| x_0^M P_i$ for
each $i=b+2,...,n.$ Since $f$ and $x_0$ are relatively prime, it follows that $f^2|P$ and $f|P_i$ for each $i$, and hence
$F\in I^2.$

Since $C$ is a local complete intersection and the ideal $I^2$ is saturated,
the conclusion now follows from Corollary 2.3 in \cite{S3}.
\end{proof}

Let $P=\binom{z+b+1}{b+1}-\binom{z-1+b}{b+1}$ (this is the Hilbert polynomial of a degree-$2$ hypersurface in $\P^{b+1}$).
Recall that $\widetilde{\Hilb}^{P}$ denotes the closure in $\Hilb^{P}$ of the set of integral $b$-dimensional closed subschemes
of degree 2; in this case, a point
in $\widetilde{\Hilb}^{P}$ is, up to a change of coordinates, a closed subscheme of the form
$V(f,x_{b+2},...,x_n)\subset\P^n,$ where
$f\in k[x_0,..,x_{b+1}]_2-\{0\}$ (not necessarily irreducible of course).
 Note that
 \begin{align}\dim \widetilde{\Hilb}^{P} &=\dim \Gr(b+1,n)+\dim \P(k[x_0,...,x_{b+1}]_2)\notag\\
 &=(b+2)n-\frac{b(b+1)}{2}.
 \label{dimbndl}
 \end{align}

By Lemma \ref{explct}, if $f\in k[x_0,...,x_{b+1}]_2-\{0\},$ then
\begin{equation}
\dim \P\left((f,x_{b+2},...,x_n)^2_l\right)=\binom{l+n}{n}-\beta_2(l)-1.
\label{dimisqr}
\end{equation}

Recall the definition of
$\widetilde{\Omega}^P\subset V(F)_{\sing}\}\subset \widetilde{\Hilb}^{P}\times\P(V)$ and
the projections $\pi$ and $\rho$ to $\widetilde{\Hilb}^{P}$ and $\P(V),$ respectively.
For $C\subset\P^n$ a closed subscheme, let $I_C$ denote its
(saturated) ideal.
Consider the subset
\[Z'=\{(C,[F])\in \widetilde{\Hilb}^{P}\times\P(V)\mid F\in I_C^2\}\subset\widetilde{\Omega}^P.\]

\begin{lemma}
The subset $Z'$ of $\widetilde{\Omega}^P$ is irreducible.
\end{lemma}

\begin{proof}
By Lemma \ref{explct},
for a fixed $f\in k[x_0,...,x_{b+1}]_2-\{0\}$ and
given $F=F_0+\sum_{i=b+2}^n F_i x_i +T\in k[x_0,...,x_n]_l,$ where
$F_0\in k[x_0,...,x_{b+1}]_l, F_i\in k[x_0,...,x_{b+1}]_{l-1}$, and $T\in (x_{b+2},...,x_n)^2_l,$ we have that
$F\in (f,x_{b+2},...,x_n)^2_l$ if and only if $f^2|F_0$ and $f|F_i$ for each $i=b+2,...,n.$

Let $V'=k[x_0,...,x_{b+1}]_{l-4}\oplus \left(\bigoplus_{i=b+2}^n k[x_0,...,x_{b+1}]_{l-3}\right)\oplus (x_{b+2},...,x_n)^2_l$.
Denote by $\A(k[x_0,...,x_{b+1}]_2)$ the affine space parametrizing points in $k[x_0,...,x_{b+1}]_2$.
Consider the composition
\[\xymatrix{
\text{Aut}(\P^n)\times\left(\A(k[x_0,...,x_{b+1}]_2)-\{0\}\right)\times\P(V')\ar[d]\\
\text{Aut}(\P^n)\times\P(k[x_0,...,x_{b+1}]_2)\times \P(V)\ar[d]\\
\widetilde{\Hilb}^{P}\times\P(V)
}
\]
where the first map is given by
\[(\sigma,f,[Q,R_{b+2},...,R_n,T])\longmapsto (\sigma, [f], [f^2Q+\sum_{i=b+2}^n fR_ix_i+T])\]
and the second map is given by
\[(\sigma, [f],[F])\longmapsto (V(f^\sigma,x_{b+2}^\sigma,...,x_n^\sigma),[F]^\sigma).\]
By construction, $Z'$ is precisely the image of the composition, hence is irreducible.
\end{proof}

\begin{remark}
It is not true that the fibers of $\widetilde{\Omega}^P\xto{\pi}\widetilde{\Hilb}^P$ are all of the same dimension.
For example, let $b=1, n=3,$ and look at $C=V(x_2^2,x_3)\in\widetilde{\Hilb}^P.$
Let $F=x_2^3x_0^{l-3}.$
Then $(C,[F])\in\pi^{-1}(C),$
but $F\notin (x_2^2,x_3)^2.$ This is why we have to study the auxiliary $Z'.$
\end{remark}

Let $Z$ be the closure of $Z'$ in $\widetilde{\Omega}^P.$

\begin{lemma} We have that
\[\dim Z=\binom{l+n}{n}-\gamma_2(l).\]
\end{lemma}

\begin{proof}
 First, $\pi(Z')=\widetilde{\Hilb}^{P},$ since given any
$C\in\widetilde{\Hilb}^{P},$ the ideal $I_C^2$ contains forms of degree $4$ already, so we can certainly find $F\in (I_C^2)_l.$
Thus, $\pi\colon Z\onto\widetilde{\Hilb}^{P}$ is onto. A generic $C\in \widetilde{\Hilb}^{P}$ is an integral
$b$-dimensional closed subscheme of degree $2$;
for such a $C$, by Lemma \ref{sq_of_ideal}, we know $Z'_C=\widetilde{\Omega}^P_C$ and hence also $Z_C=Z'_C$. This allows us to
compute $\dim Z_C=\dim Z'_C=\binom{l+n}{n}-\beta_2(l)-1.$ This computes $\dim Z=\dim\widetilde{\Hilb}^{P}+\dim Z_C$
and gives the desired result, by virtue of (\ref{dimbndl}) and (\ref{dimisqr}).
\end{proof}

\begin{lemma}
$X^2:=\rho(Z)$ is an irreducible closed subset of $X$ of dimension $\binom{l+n}{n}-\gamma_2(l)$.
If $[F]\in X$ contains an integral closed subscheme of dimension $b$ and degree 2
in its singular locus, then $[F]\in X^2.$
\end{lemma}

\begin{proof}
It is clear that $\rho(Z)$ is an irreducible closed subset of $X$, since $Z$ is irreducible and closed in
$\widetilde{\Omega}^P$. Choose any integral $b$-dimensional $C$ of degree $2$. Apply Lemma \ref{ctrlsng} to $C$ to find
$[F]\in \P(V)$ such that
we have a homeomorphism $C\into V(F)_{\sing}.$ If $\hat{C}\in\widetilde{\Hilb}^P$ is another closed subscheme contained in
$V(F)_{\sing},$ then necessarily we have $C\into\hat{C},$ since $C$ is reduced. Hence
$C=\hat{C},$ since $C$ and $\hat{C}$ have the same Hilbert
polynomial. Therefore, the map
$Z\to\rho(Z)$ has a $0$-dimensional
fiber, so $\dim\rho(Z)=\dim Z$.

Let $[F]\in X$ be such that $V(F)_{\sing}$ contains an
integral $b$-dimensional closed subscheme $C$ of $\P^n$
of degree $2$.  Then we know that $F\in I_C^2$ by Lemma \ref{sq_of_ideal},
so $(C,[F])\in Z',$ and hence in fact $[F]\in \rho(Z')\subset\rho(Z)=X^2.$
\end{proof}

\begin{remark}
Lemma \ref{scompsmldeg} did not treat the case $b=n-1, d=3.$ We discuss this now. When $b=n-1,$ we can describe $X$ explicitly.
Indeed, if $V(G)$ is an integral $(n-1)$-dimensional closed subscheme of $\P^n_k$
(here $k$ has any characteristic) with
$V(G)\subset V(F)_{\sing},$ then necessarily $F=G^2H$ for some $H$ (since $V(G)$ is a complete intersection and the ideal
$(G^2)$ is saturated; see Corollary 2.3 in \cite{S3}). For $d=1,...,\lfloor\frac{l}{2}\rfloor,$ consider the map
\begin{align*}
\varphi_{d}\colon \P(k[x_0,...,x_n]_d)\times\P(k[x_0,...,x_n]_{l-2d}) &\longrightarrow \P(k[x_0,...,x_n]_l)\\
(G,H) &\longmapsto G^2H.
\end{align*}
Certainly, $\text{im}(\varphi_d)\subset T^d_k\subset X$ and $X=\bigcup_{d=1}^{\lfloor\frac{l}{2}\rfloor}\text{im}(\varphi_d),$ so
\[X=X^1\cup\text{im}(\varphi_2)\cup\text{im}(\varphi_3)   \cup\left(\bigcup_{d=4}^{\lfloor\frac{l}{2}\rfloor}T^d\right).\]
Since any point in the image of $\varphi_d$ has only finitely many preimages, it follows that
\[\dim\text{im}(\varphi_d)=\binom{d+n}{n}+\binom{l-2d+n}{n}-2.\]
So $\dim\text{im}(\varphi_3)<\dim\text{im}(\varphi_2)=\dim X^2$ for $l\geq l_0$ (where $l_0$ is effectively computable) and hence when
$b=n-1,$ it suffices to bound $\dim T^d_k$ only for $d\geq 4,$ which was handled by Lemma \ref{scompsmldeg}.
\label{techdetscomprem}
\end{remark}

\begin{proof}[Proof of Theorem \ref{corscomp}]
Let $k=\overline{\FF_p}.$
With the above preparations, the proof is now analogous to that of Theorem \ref{mainthm}.
We use Lemma \ref{scompsmldeg} with $B=2p^b(n-b+1)$ (or Remark \ref{techdetscomprem} if $b=n-1$) and Proposition \ref{lrg_m} applied to the triple
$(l,B+1,\gamma_2(l)+1)$ to argue that
if $Z\subset T^d_k$ is an irreducible component of $T^d_k$ (where $d\geq 3$), then either $Z\subset T^1_k\cup T^2_k,$ or $\dim Z<\dim X^2$
(as long as $l\geq l_0,$ for some effectively computable $l_0$).

We have
\[X=\bigcup_{d=1}^N T^d_k\qquad\text{for $N=l(l-1)^{N+1}$.}.\]
If $Z$ is an irreducible component of $X$ with $\dim Z\geq \dim X^2,$ then $Z\subset T^d_k$ for some $d.$ If $d\geq 3,$ then by the
previous paragraph, we have $Z\subset T^1\cup T^2.$ So in any case, $Z\subset T^1\cup T^2=X^1\cup X^2.$ Hence $Z=X^1$ or $Z=X^2$.
\end{proof}

\begin{remark}
Let $p\neq 2.$ If we could prove that $\dim T^d_{\overline{\FF_p}}<\dim X^2$ for all $d\geq 3,$ we would be able to deduce that
for $d\geq 3,$
\[\dim T^d_{\overline{\Q}}\leq\dim T^d_{\overline{\FF_p}}<\dim X^2.\]

Suppose instead that $\dim T^d_k\geq\dim X^2$ for some $d\geq 3$ and $k=\overline{\FF_p}.$ Let $Z$ be an irreducible component
of $T^d_k$ with $\dim Z\geq \dim X^2$.
We have $Z\subset X^1\cup X^2$ by the proof of Corollary \ref{corscomp}. Moreover,
$Z\nsubseteq X^2$ (since $X^2\nsubseteq T^d_k$ by Corollary \ref{ingforuniqseccomp}), so $Z\subset X^1.$ So it would suffice to
prove that $\dim (T^d_k\cap X^1)<\dim X^2$ for $d\geq 3$ (this inequality fails when $d=2$).
This is the technical
problem that unfortunately does not allow us to remove the assumption $\ch k\neq 0$ from Theorem \ref{corscomp}.

\label{techcompl}
\end{remark}

\section*{Acknowledgments}

The current paper consists of my doctoral thesis at MIT, and I would like to thank first and foremost my advisor Bjorn Poonen for
his invaluable help. He guided me with dedication throughout the entire process, starting with the formulation of my thesis problem, going through numerous suggestions for ideas and approaches, and finishing with feedback for editing the final write-up.
In particular, the main idea used in the case of large degree, which is the heart of the current paper, belongs to him, and I deeply admire Prof. Poonen's academic generosity.  I would also like to thank Joe Harris
for the fruitful discussions, and particularly for teaching me the specialization argument, which is one of the two key arguments in the paper. I am grateful to Dennis Gaitsgory for his inspiring algebraic geometry lectures that helped me for the technical part of the paper. I also thank the referee for a number of valuable suggestions.

\bibliography{main}
\bibliographystyle{plain}

\end{document}